\documentclass[12pt,reqno]{amsart}

\usepackage{amssymb}
\usepackage{amsthm}
\usepackage{amsmath}
\usepackage{a4wide}
\usepackage{latexsym}
\usepackage{mathrsfs}
\usepackage[T1]{fontenc}
\usepackage[latin1]{inputenc}
\usepackage{enumitem}
\usepackage{color}

\newtheorem{theorem}{Theorem}[section]
\newtheorem{lemma}[theorem]{Lemma}
\newtheorem{proposition}[theorem]{Proposition}
\newtheorem{corollary}[theorem]{Corollary}

\theoremstyle{definition}
\newtheorem{definition}[theorem]{Definition}

\numberwithin{equation}{section}

\newcommand{\N}{\mathbb{N}}

\newcommand{\R}{\mathbb{R}}


\begin{document}
\title[The algebras $\mathscr{B}(T)$ and $\mathscr{B}(B_p)$ are not Grothendieck spaces]{The algebras of bounded operators on the Tsirelson and Baernstein spaces are not Grothendieck spaces}

\author[K.~Beanland]{Kevin Beanland}
\address{Department of Mathematics, Washington and Lee University, Lexington, VA 24450, USA}
\email{beanlandk@wlu.edu}

\author[T.~Kania]{Tomasz Kania}
\address{Mathematics Institute,
University of Warwick,
Gibbet Hill Rd, 
Coventry, CV4 7AL, 
United Kingdom}
\email{tomasz.marcin.kania@gmail.com}

\author[N.~J.~Laustsen]{Niels Jakob Laustsen}
\address{Department of Mathematics and Statistics, Fylde College, Lancaster University, Lancaster,
LA1 4YF, United Kingdom}
\email{n.laustsen@lancaster.ac.uk}
\begin{abstract}
We present two new examples of reflexive Banach spaces~$X$ for which the associated Banach algebra~$\mathscr{B}(X)$ of bounded operators on~$X$ is not a Grothen\-dieck space, namely $X = T$ (the Tsirelson space) and $X = B_p$ (the $p^{\text{th}}$ Baern\-stein space) for $1<p<\infty$.
\end{abstract}
\maketitle
\section{Introduction and statement of the main result}
\noindent A \emph{Grothen\-dieck space} is a Banach space $X$ for which every weak*-convergent sequence in the dual space $X^*$ converges weakly. The name originates from a result of Grothen\-dieck, who showed that $\ell_\infty$, and more generally every injective Banach space, has this property. Plainly, every reflexive Banach space~$X$ is a Grothen\-dieck space because the weak and weak* topologies on $X^*$ coincide. By the Hahn--Banach theorem, the class of Grothen\-dieck spaces is closed under quotients, and hence in particular under passing to complemented sub\-spaces.\smallskip

Substantial efforts have been devoted to the study of Grothen\-dieck spaces over the years, especially in the case of $C(K)$-spaces. Notable achievements include the constructions by Ta\-la\-grand~\cite{tala} (assuming the Continuum Hypo\-thesis) and Haydon~\cite{haydon} of compact Haus\-dorff spaces~$K_T$ and~$K_H$, respectively, such that the Banach spaces~$C(K_T)$ and~$C(K_H)$ are Grothen\-dieck, but~$C(K_T)$  has no quotient isomorphic to~$\ell_\infty$, while~$C(K_H)$ has the weaker property that it contains no subspace isomorphic to~$\ell_\infty$. However, $K_H$ has the significant advantage over~$K_T$ that it exists within ZFC. Haydon, Levy and Odell~\cite[Corollary~3F]{hlo} have subsequently shown that Ta\-la\-grand's result cannot be obtained within ZFC itself because under the assumption of Martin's Axiom and the negation of the Continuum Hypo\-thesis, every non-reflexive Grothen\-dieck space has a quotient isomorphic to~$\ell_\infty$. Brech~\cite{brech} has pursued these ideas even further using forcing to construct a model of ZFC in which there is a compact Haus\-dorff space~$K_B$ such that~$C(K_B)$ is a Grothen\-dieck space and has density strictly smaller than the continuum, so in particular no quotient of~$C(K_B)$ is isomorphic to~$\ell_\infty$. In another direction, Bour\-gain~\cite{bourgain} has shown that~$H^\infty$ is a Grothen\-dieck space.\smallskip 

Nevertheless, a general structure theory of Grothen\-dieck spaces is yet to materialize, and many fundamental questions about the nature of this class remain open. Diestel \cite[\S{}3]{diestel1973} produced an expository list of such questions in 1973. It is remarkable how few of these questions that have been resolved in the meantime.\smallskip

We shall make a small contribution towards the resolution of the seemingly very difficult problem of describing the Banach spaces~$X$ for which the associated Banach algebra~$\mathscr{B}(X)$ of bounded operators on~$X$ is a Grothen\-dieck space by proving the following result. 
\begin{theorem}\label{mainthm}
Let~$X$ be either the Tsirelson space~$T$ or the $p^{\text{th}}$ Baern\-stein space~$B_p$, where $1<p<\infty$. Then the Banach algebra~$\mathscr{B}(X)$ is not a Grothen\-dieck space. 
\end{theorem}

For details of the spaces~$T$ and~$B_p$, we refer to Section~\ref{section3}. We remark that  these spaces are not the first examples of reflexive Banach spaces~$X$ for which~$\mathscr{B}(X)$ is known not to be a Grothen\-dieck space due to the following result of the second-named author~\cite{kania}.
\begin{theorem}[Kania]\label{kania}
Let $X = \bigl(\bigoplus_{n\in\N} \ell_q^n\bigr)_{\ell_p}$, where $1<p<\infty$ and either $q = 1$ or $q=\infty$. Then the Banach algebra~$\mathscr{B}(X)$ is not a Grothen\-dieck space. 
\end{theorem}

This paper is organized as follows: in Section~\ref{section3}, we prove Theorem~\ref{mainthm}, followed by a~discussion of the key difference between the two cases (see Proposition~\ref{unbdddiagop} and the paragraph preceding it for details), before we conclude with a short section listing some related open problems.

\section{The proof of Theorem~\ref{mainthm}}\label{section3}
\noindent
The proof that $\mathscr{B}(T)$ is not a Grothen\-dieck space relies on abstracting  the strategy used to prove Theorem~\ref{kania}.
The following notion will play a key role in this approach. Let~$E$ be a Banach space
with a normalized, $1$-unconditional basis $(e_n)_{n=1}^\infty$ (where `$1$-un\-con\-di\-tion\-al' means that $\bigl\|\sum_{j=1}^n\alpha_j\beta_j e_j\bigr\|\leqslant \max_{1\leqslant j\leqslant n}|\alpha_j|\cdot\bigl\|\sum_{j=1}^n\beta_j e_j\bigr\|$ for each $n\in\N$ and all choices of scalars $\alpha_1,\ldots,\alpha_n,\beta_1,\ldots,\beta_n$.) 
The $E$-\emph{direct sum} of a sequence $(X_n)_{n=1}^\infty$ of Banach spaces is  given by
\[ \Bigl(\bigoplus_{n\in\N} X_n\Bigr)_E = \biggl\{ (x_n) : x_n\in X_n\ (n\in\N)\ \text{and the series}\ \sum_{n=1}^\infty \|x_n\| e_n\ \text{converges in}\ E\biggr\}. \]
This is a Banach space with respect to the coordinate-wise defined operations and the norm
\[ \| (x_n) \|  = \biggl\|\sum_{n=1}^\infty \|x_n\| e_n\biggr\|. \]
Analogously, we write $\bigl(\bigoplus_{n\in\N} X_n\bigr)_{\ell_\infty}$ for the Banach space of uniformly bounded se\-quences $(x_n)$ with $x_n\in X_n$ for each $n\in\N$, equipped with the coordinate-wise defined operations and the norm $\| (x_n)\| = \sup_n\|x_n\|$.\smallskip

The main property of the $E$-direct sum that we shall require is that every uniformly bounded sequence $(U_n)_{n=1}^\infty$ of operators, where $U_n\in\mathscr{B}(X_n)$ for each $n\in\N$, induces a~`diagonal operator'~$\operatorname{diag}(U_n)$ on $\bigl(\bigoplus_{n\in\N} X_n\bigr)_E$ by the definition
\[ \operatorname{diag}(U_n) (x_n)_{n=1}^\infty = (U_nx_n)_{n=1}^\infty, \]
and $\|\operatorname{diag}(U_n)\| = \sup_n\|U_n\|$.\smallskip

We shall also use the following result of W.~B.~Johnson~\cite{johnson}.

\begin{theorem}[Johnson]\label{johnsonell1thm}
The Banach space~$\bigl(\bigoplus_{n\in\N}\ell_1^n\bigr)_{\ell_\infty}$ contains a complemented copy
of~$\ell_1$. Hence a Banach space that contains a~complemented copy of~$\bigl(\bigoplus_{n\in\N}\ell_1^n\bigr)_{\ell_\infty}$ is not a~Grothen\-dieck space. 
\end{theorem}

\begin{lemma}\label{lemmaKaniaAbstractMethod}
Let $X = \bigl(\bigoplus_{n\in\N} X_n\bigr)_E$, where~$E$ is a Banach space with a normalized, $1$\nobreakdash-un\-condi\-tional basis and $(X_n)$ is a sequence of Banach spaces. Then $\mathscr{B}(X)$ contains a complemented sub\-space which is isometrically isomorphic to~$\bigl(\bigoplus_{n\in\N} X_n\bigr)_{\ell_\infty}$.
\end{lemma}
\begin{proof} We may suppose that $X_n$ is non-zero for each $n\in\N$. Let $n\in\N$, and take $w_n\in X_n$ and $f_n\in\ X_n^*$ such that $\|                                                    
w_n\| = \| f_n\| = 1 = \langle w_n, f_n\rangle$.  For $x_n\in X_n$,
the rank-one operator given by
\[ x_n\otimes f_n\colon y\mapsto \langle y,f_n\rangle x_n,\quad X_n\to X_n, \]
has the same norm as~$x_n$, so the map
\begin{equation}\label{defnDeltaT} \Delta\colon\ (x_n)\mapsto \operatorname{diag}(x_n\otimes f_n),\quad \Bigl(\bigoplus_{n\in\N} X_n\Bigr)_{\ell_\infty}\to\mathscr{B}(X), \end{equation}
is an isometry. It is clearly linear, and therefore the image of~$\Delta$ is
a subspace of~$\mathscr{B}(X)$ which is isometrically isomorphic
to~$\bigl(\bigoplus_{n\in\N} X_n\bigr)_{\ell_\infty}$.  This
subspace is complemented in~$\mathscr{B}(X)$ because~$\Delta$ has a bounded, linear left
inverse, namely the map given by
\[ U\mapsto (Q_nUJ_nw_n)_{n=1}^\infty,\quad
\mathscr{B}(X)\to\Bigl(\bigoplus_{n\in\N} X_n\Bigr)_{\ell_\infty}, \] where $J_n\colon X_n\to X$ and $Q_n\colon X\to X_n$ denote the  $n^{\text{th}}$ coordinate embedding and projection, respectively. 
\end{proof}

\begin{corollary}\label{corKaniaAbstractMethod} Let~$X$ be a Banach space which contains a complemented subspace that is iso\-mor\-phic to $\bigl(\bigoplus_{n\in\N} \ell_1^{m_n}\bigr)_E$ for some unbounded sequence $(m_n)$ of natural numbers and some Banach space~$E$ with a normalized, $1$-unconditional basis. Then $\mathscr{B}(X)$ is not a~Gro\-then\-dieck space. 
\end{corollary}

\begin{proof}
Let $Y= \bigl(\bigoplus_{n\in\N} \ell_1^{m_n}\bigr)_E$. Lemma~\ref{lemmaKaniaAbstractMethod} implies that~$\mathscr{B}(Y)$
contains a complemented copy of~$\bigl(\bigoplus_{n\in\N}\ell_1^{m_n}\bigr)_{\ell_\infty}$, which is isomorphic to~$\bigl(\bigoplus_{n\in\N}\ell_1^n\bigr)_{\ell_\infty}$ by Pe\l{}czy\'{n}ski's decomposition method, and therefore~$\mathscr{B}(Y)$ is not a Grothen\-dieck space by Theorem~\ref{johnsonell1thm}. The assump\-tion means that we can find bounded operators \mbox{$U\colon X\to Y$} and $V\colon Y\to X$ such that $UV = I_Y$. This implies that~$\mathscr{B}(X)$ contains a complemented copy of~$\mathscr{B}(Y)$ because the operator \mbox{$R\mapsto URV,$} \mbox{$\mathscr{B}(X)\to\mathscr{B}(Y),$} is a left inverse of \mbox{$S\mapsto VSU,$} \mbox{$\mathscr{B}(Y)\to\mathscr{B}(X)$}. There\-fore~$\mathscr{B}(X)$ is  not a Grothen\-dieck space.
\end{proof}

Following Figiel and Johnson~\cite{FJ}, we use the term `the Tsirelson space' and the symbol~$T$ to denote the dual of the reflexive Banach space originally constructed by Tsirelson~\cite{T} with the property that it does not contain any of the classical sequence spaces~$c_0$ and~$\ell_p$ for $1\leqslant p<\infty$. We refer to~\cite{CS} for an attractive introduction to the Tsirelson space, including background information, its formal definition and a comprehensive account of what was known about it up until the late 1980's.\smallskip  

The following notion plays a key role in the study of the Tsirelson space (and in the definition of the Baern\-stein spaces, to be given below). 
\begin{definition} A non-empty, finite subset $M$ of~$\N$ is  \emph{(Schreier-)admissible} if \mbox{$|M|\leqslant\min M$}, where $|M|$ denotes the cardinality of~$M$. 
\end{definition}

\begin{proof}[Proof of Theorem~{\normalfont{\ref{mainthm}}} for $X=T$]  
The unit vector basis is a normalized, $1$-un\-condi\-tional basis for~$T$. We shall denote it by $(t_n)_{n = 1}^\infty$ throughout this proof.
Take natural numbers \mbox{$1=m_1\leqslant k_1<m_2\leqslant k_2<m_3\leqslant k_3<\cdots$}, 
and set 
\[ M_n = [m_n,m_{n+1})\cap\N\qquad\text{and}\qquad 
   F_n = \operatorname{span}\{ t_j : j\in M_n\}\qquad (n\in\N), \]
so that $(F_n)$ is an unconditional finite-dimensional Schauder decomposition of~$T$. 
For $x_n\in F_n\ (n\in\N)$, \cite[Corollary~7(i)]{CJT} shows that the series $\sum_{n=1}^\infty x_n$ converges in~$T$ if and only if the series $\sum_{n=1}^\infty \|x_n\| t_{k_n}$ converges in~$T$, and when they both converge, the norms of their sums are related by
\[ \frac{1}{3}\,\biggl\|\sum_{n=1}^\infty \|x_n\| t_{k_n}\biggr\|\leqslant
\biggl\| \sum_{n=1}^\infty x_n\biggr\|\leqslant 18\,\biggl\|\sum_{n=1}^\infty
\|x_n\| t_{k_n}\biggr\|. \] Consequently~$T$ is $54$-isomorphic to the direct sum $\bigl(\bigoplus_{n\in\N} F_n\bigr)_E$, where $E$ denotes the closed linear span of $\{ t_{k_n} : n\in\N\}$.
Taking $m_n = 2^{n-1}$ for each $n\in\N$, we see that the sets $M_n$ are admissible, which implies that~$F_n$ is $2$-isomorphic to~$\ell_1^{m_n}$ for each $n\in\N$. Hence~$T$ is $108$-isomorphic to $\bigl(\bigoplus_{n\in\N} \ell_1^{m_n}\bigr)_E$, and the conclusion follows from Corollary~\ref{corKaniaAbstractMethod}.
\end{proof}
We shall now turn our attention to the Baern\-stein spaces~$B_p$ for $1<p<\infty$. As the name suggests, they originate in the work of Baern\-stein, who introduced the space that we call~$B_2$ in~\cite{baern}, while the  variant for general~$p$ is due to Seifert~\cite{seifert}. 
These spaces can be viewed as a natural precursor of the Tsirelson space, as the account in~\cite[Chapter 0]{CS} highlights. \smallskip

The main reason for our interest in the Baern\-stein spaces in the present context is that while the conclusion of Theorem~\ref{mainthm} remains valid for them, the method of proof that we used to establish Theorem~\ref{mainthm} for the Tsirelson space does not carry over. We shall return to this point in Proposition~\ref{unbdddiagop} below, after we have given the formal definition of the Baern\-stein spaces and shown how to deduce Theorem~\ref{mainthm} for them.\smallskip

In the remainder of this section, we fix a number $p\in(1,\infty)$. For $x = (\alpha_n)_{n=1}^\infty\in c_{00}$ (the vector space of finitely supported scalar sequences), $k\in\N$ and $N_1,\ldots, N_k\subseteq\N$, let 
\[ \nu_p(x; N_1,\ldots, N_k) = \biggl(\sum_{j=1}^k\mu(x,N_j)^p\biggr)^{\frac{1}{p}},\qquad\text{where}\qquad \mu(x,N_j) = \sum_{n\in N_j} |\alpha_n|. \] 
Given two non-empty subsets $M$ and $N$ of~$\N$, where $M$ is finite, we use the notation $M<N$ to indicate that $\max M<\min N$.  
The $p^{\text{th}}$ \emph{Baern\-stein space}~$B_p$ can now be defined as the completion of~$c_{00}$ with respect to the norm
\begin{multline}\label{defnBaernsteinnorm} \| x\|_{B_p} = \sup\{ \nu_p(x;N_1,\ldots, N_k) : k\in\N\ \text{and}\\ N_1<N_2<\cdots<N_k\ \text{are admissible subsets of}\ \N\}. \end{multline}

As we have already mentioned, Baern\-stein introduced the Banach space~$B_2$ in~\cite{baern} and observed that it is reflexive and the unit vector basis of~$c_{00}$, which we shall here denote by~$(b_n)_{n=1}^\infty$, is a normalized,  $1$-unconditional basis for it. His proofs carry over immediately to general~$p$. We write $(b_n^*)_{n=1}^\infty$ for the sequence of coordinate functionals corresponding to the basis~$(b_n)_{n=1}^\infty$.\smallskip 

In analogy with the proof of Theorem~\ref{mainthm} for the Tsirelson space given above, we shall define $m_n = 2^{n-1}$ and $M_n = [m_n,m_{n+1})\cap\N$ for each $n\in\N$, and we shall then consider the finite-dimensional blocking 
\begin{equation}\label{defnFn} F_n = \operatorname{span}\{ b_j : j\in M_n\}\qquad (n\in\N) \end{equation}
of the unit vector basis for~$B_p$. For later reference, we remark that $F_n$ is isometrically isomorphic to~$\ell_1^{m_n}$ because $M_n$ is admissible for each $n\in\N$.\smallskip  

Let $\bigl(\bigoplus_{n\in\N} F_n\bigr)_{c_{00}}$ denote the vector space of sequences $(x_n)$ such that $x_n\in F_n$ for each $n\in\N$ and $x_n = 0$ eventually. For  $(x_n)\in\bigl(\bigoplus_{n\in\N} F_n\bigr)_{c_{00}}$, we set
\begin{equation}\label{defnDelta} \Delta(x_n) = \sum_{n=1}^\infty x_n\otimes
  b^*_{m_{n+1}-1}, \end{equation} 
where we note that the sum is finite, so that~$\Delta(x_n)$ defines a finite-rank operator on~$B_p$.

\begin{lemma}\label{normestimateofDelta} For each $(x_n)\in\bigl(\bigoplus_{n\in\N} F_n\bigr)_{c_{00}},$  
\[ \max_{n\in\N}\| x_n\|_{B_p}\leqslant \|\Delta(x_n)\|\leqslant\sqrt[p]{3}\max_{n\in\N}\| x_n\|_{B_p}. \]
\end{lemma}

\begin{proof}
The lower bound is clear because, for each $k\in\N$, $b_{m_{k+1}-1}$ is a unit vector in~$B_p$, and $\Delta(x_n)b_{m_{k+1}-1} = x_k$.\smallskip  

It suffices to establish the upper bound in the case where $\max_{n\in\N}\| x_n\|_{B_p} = 1$. Then $\mu(x_n,N)\leqslant 1$ for each $n\in\N$ and $N\subseteq\N$ because the support of~$x_n$ is admissible.\smallskip   

For $N\subseteq\N$, we introduce the set
\[ \operatorname{uep}(N) = \{m_{k+1}-1 : k\in\N,\,N\cap M_k\ne\emptyset\}, \] 
which consists of the upper end points (hence the acronym `uep') of the intervals~$M_k$ that $N$ intersects. For later reference, we remark that $\operatorname{uep}(N)$ is ad\-missible whenever~$N$ is admissible because
\[ \min\operatorname{uep}(N)\geqslant\min N\geqslant |N|\geqslant|\!\operatorname{uep}(N)|. \]

Suppose that $N\subseteq\N$ is non-empty and finite, and take a non-empty, finite subset~$K$ of~$\N$
such that $\operatorname{uep}(N) = \{ m_{k+1}-1 : k\in K\}$. Then~$N$ is the disjoint union of the family $\{ N\cap M_k : k\in K\}$, and for each $y\in B_p$, we have
\begin{align}
\mu(\Delta(x_n)y,N) &= \sum_{k\in K}\mu(\Delta(x_n)y,N\cap M_k) = \sum_{k\in K}|\langle y,b_{m_{k+1}-1}^*\rangle|\,\mu(x_k, N\cap M_k)\notag\\
&\leqslant \sum_{k\in K}|\langle y,b_{m_{k+1}-1}^*\rangle| = \mu(y,\operatorname{uep}(N)).\label{BnGstep4} 
\end{align}

Now let $N_1<N_2<\cdots<N_k$ be admissible subsets of~$\N$ for some $k\in\N$.  We aim to show that~$3\|y\|_{B_p}^p$ is an upper bound of the quantity
\begin{equation}\label{BnGeq5} \nu_p(\Delta(x_n)y; N_1,\ldots,N_k)^p = \sum_{j=1}^k\mu(\Delta(x_n)y,N_j)^p. \end{equation} 
By adding at most two extra sets beyond the final set~$N_k$, we may suppose that~$k$ is a~multiple of~$3$. Moreover, we may suppose that there is no $j\in\{1,\ldots,k-1\}$ such that $N_j\cup N_{j+1}\subseteq M_h$ for some $h\in\N$. Indeed, if there is, we can replace the sets $N_j$ and $N_{j+1}$ with their union $N_j\cup N_{j+1}$, which is still admissible, and this change will not decrease the value of~\eqref{BnGeq5} because
\begin{align*} \mu(\Delta(x_n)y,N_j)^p + \mu(\Delta(x_n)y,N_{j+1})^p &\leqslant 
\bigl(\mu(\Delta(x_n)y,N_j) + \mu(\Delta(x_n)y,N_{j+1})\bigr)^p\\ &=
\mu(\Delta(x_n)y, N_j\cup N_{j+1})^p. \end{align*}

Set $r_j = \min\operatorname{uep}(N_j)$ and $s_j = \max\operatorname{uep}(N_j)$ for each $j\in\{1,\ldots,k\}$. Then we
have $r_1\leqslant s_1\leqslant r_2\leqslant s_2\leqslant\cdots\leqslant r_k\leqslant s_k$ because
$N_1<N_2<\cdots<N_k$. A much less obvious fact is that $s_j<r_{j+3}$ for
each $j\in\{1,\ldots,k-3\}$ (provided that $k>3$). To verify this, we assume the contrary, so that $s_j=r_{j+3} = m_{h+1}-1$ for some $j\in\{1,\ldots,k-3\}$ and $h\in\N$. Then it follows that~$N_{j+1}$ and~$N_{j+2}$ are both contained
in~$M_h$, which contradicts the above assumption. In other words, we have
\[ \operatorname{uep}(N_1)<\operatorname{uep}(N_4)<\cdots<\operatorname{uep}(N_{k-2}),\qquad
\operatorname{uep}(N_2)<\operatorname{uep}(N_5)<\cdots<\operatorname{uep}(N_{k-1}) \]
and
\[ \operatorname{uep}(N_3)<\operatorname{uep}(N_6)<\cdots<\operatorname{uep}(N_{k}), \]
where each of the sets $\operatorname{uep}(N_j)$ is admissible. Hence, using~\eqref{BnGstep4}, we conclude that
\begin{align*}
\nu_p(\Delta(x_n)y; N_1,\ldots,N_k)^p &\leqslant
\sum_{j=1}^k\mu(y,\operatorname{uep}(N_j))^p\\ &= \nu_p(y;
\operatorname{uep}(N_1),\operatorname{uep}(N_4),\ldots,\operatorname{uep}(N_{k-2}))^p\\ &\qquad +
\nu_p(y;\operatorname{uep}(N_2),\operatorname{uep}(N_5),\ldots,\operatorname{uep}(N_{k-1}))^p\\ 
&\qquad + \nu_p(y;\operatorname{uep}(N_3),\operatorname{uep}(N_6),\ldots,\operatorname{uep}(N_{k}))^p\\ &\leqslant
3\,\|y\|_{B_p}^p,
\end{align*}
as required.
\end{proof}

\begin{corollary}\label{c0ofl1nsinKB}
The ideal $\mathscr{K}(B_p)$ of compact operators on~$B_p$ contains a complemented sub\-space that is isomorphic to~$\bigl(\bigoplus_{n\in\N}\ell_1^n\bigr)_{c_0}$.
\end{corollary}

\begin{proof}
The map 
\begin{equation}\label{defnDeltaBp}
\Delta\colon\ (x_n)\mapsto\Delta(x_n),\quad \Bigl(\bigoplus_{n\in\N} F_n\Bigr)_{c_{00}}\to\mathscr{K}(B_p), \end{equation}
given by~\eqref{defnDelta} is linear, and Lemma~\ref{normestimateofDelta} implies that it is bounded with respect to the norm $\|(x_n)\|_\infty = \max_{n\in\N}\|x_n\|_{B_p}$ on its domain, so it extends  uniquely to a bounded operator \[ \Delta\colon \Bigl(\bigoplus_{n\in\N} F_n\Bigr)_{c_0}\to\mathscr{K}(B_p). \]

For $n\in\N$, let $Q_n\colon B_p\to F_n$ be the canonical basis projection of~$B_p$ onto~$F_n$, and define 
\[ \Theta(U) = (Q_nUb_{m_{n+1}-1})_{n=1}^\infty\qquad (U\in\mathscr{K}(B_p)). \]
Since $B_p$ is reflexive, $(b_{m_{n+1}-1})_{n=1}^\infty$ is a weak-null sequence. Hence  $(Ub_{m_{n+1}-1})_{n=1}^\infty$ is norm-null for each compact operator~$U$, and therefore $U\mapsto \Theta(U)$ defines a map from~$\mathscr{K}(B_p)$ into~$\bigl(\bigoplus_{n\in\N} F_n\bigr)_{c_0}$. This map is clearly bounded and linear, and it is a left inverse  of~$\Delta$. Now the conclusion follows from the facts that $F_n$ is isometrically isomorphic to~$\ell_1^{m_n}$ for each $n\in\N$ and $\bigl(\bigoplus_{n\in\N}\ell_1^{m_n}\bigr)_{c_0}$ is isomorphic to $\bigl(\bigoplus_{n\in\N}\ell_1^n\bigr)_{c_0}$. \end{proof}

\begin{proof}[Proof of Theorem~{\normalfont{\ref{mainthm}}} for $X=B_p$]
The bidual of~$\mathscr{K}(B_p)$ is~$\mathscr{B}(B_p)$ because~$B_p$ is reflexive and has a basis. Hence, passing to the biduals in Corollary~\ref{c0ofl1nsinKB}, we see that~$\mathscr{B}(B_p)$ contains a complemented subspace that is isomorphic to~$\bigl(\bigoplus_{n\in\N}\ell_1^n\bigr)_{\ell_\infty}$, and the conclusion follows from Theorem~\ref{johnsonell1thm} as before.
\end{proof}

Comparing the above proofs of Theorem~\ref{mainthm} for the Tsirelson  space on the one hand and the $p^{\text{th}}$ Baern\-stein space on the other, we see that the former is significantly  shorter and simpler. This is due to the fact that~$T$ is isomorphic to the $E$-direct sum of the blocks~$F_n$ for a suitably chosen Banach space~$E$ with a normalized, $1$-unconditional basis. Indeed, this fact immediately allowed us to define the diagonal operator~\eqref{defnDeltaT} for~$T$, whereas it required substantial work to establish the boundedness of its counterpart, which is the bi\-dual of the operator~\eqref{defnDeltaBp}, for~$B_p$.  One may wonder whether this extra effort is really necessary. The following result addresses this question, showing that at least some argument is required.

\begin{proposition}\label{unbdddiagop}
There exists a uniformly bounded sequence $(U_n)_{n=1}^\infty$ of rank-one op\-er\-a\-tors, where $U_n$ is defined on the subspace~$F_n$ of~$B_p$ given by~\eqref{defnFn}, such that the corresponding diagonal map defined on the subspace $\operatorname{span}\bigcup_{n\in\N} F_n\ (=c_{00})$ of~$B_p$ by 
\begin{equation}\label{diagop} \operatorname{diag}(U_n)\biggl(\sum_{j=1}^k x_j\biggr) = \sum_{j=1}^k U_jx_j\qquad (k\in\N,\, x_1\in F_1,\ldots, x_k\in F_k) \end{equation}
is not bounded with respect to the norm~$\|\cdot\|_{B_p}$.  
\end{proposition}

It follows in particular that we cannot express~$B_p$ as the $E$-direct sum of the blocks~$F_n$ for any Banach space~$E$ with a normalized, $1$-unconditional basis, so that there is no counter\-part of \cite[Corollary~7(i)]{CJT} for~$B_p$. More generally, Proposition~\ref{unbdddiagop} implies that `diagonal operators' need not exist if one replaces the $E$-direct sum $X= \bigl(\bigoplus_{n\in\N} X_n\bigr)_E$ for a Banach space~$E$ with a normalized, $1$-unconditional basis with a Banach space~$X$ that merely has an unconditional finite-dimensional Schauder decomposition $(X_n)_{n=1}^\infty$.\smallskip  
 
In the proof of Proposition~\ref{unbdddiagop},  we shall require the following elementary inequality.
\begin{lemma}\label{calclemma}
Let $a,b,c\in(0,\infty)$ with $a>c$, and let $p\in(1,\infty)$. Then
\[ a^p + (b+c)^p < (a+b)^p + c^p. \]
\end{lemma} 
\begin{proof}
Consider the function $f\colon (0,\infty)\to\R$ given by $f(t) = (t+d)^p + 1 - t^p - (d+1)^p$, where $d\in(0,\infty)$ is a constant. Since~$f$ is differentiable with  \mbox{$f'(t) = p(t+d)^{p-1} - pt^{p-1}>0$} and $f(1)=0$, we conclude that $f(t)>0$ for $t>1$. Now the result follows by taking $t = a/c>1$ and $d = b/c>0$ in this inequality and rearranging it.  
\end{proof} 
 
\begin{proof}[Proof of Proposition~{\normalfont{\ref{unbdddiagop}}}] 
For each $n\in\N$, let $U_n$ be the summation operator onto the first coordinate of~$F_n$, that is, \[ U_n\colon\  \sum_{j\in M_n} \alpha_j b_j\mapsto \biggl(\sum_{j\in M_n} \alpha_j\biggr) b_{m_n},\quad F_n\to F_n. \] This is a rank-one operator which has norm~$1$ because $F_n$ is isometrically isomorphic to~$\ell_1^{m_n}$.\smallskip  

Take natural numbers $q\leqslant r$, and set 
\[ y_{q,r} = \sum_{n=q}^{r} \frac{1}{n}x_n,\qquad\text{where}\qquad x_n = \frac{1}{2^{n-1}}\sum_{j\in M_n} b_j\in F_n\qquad (n\in\N). \]
We aim to prove that
\begin{equation}\label{normypq} \| y_{q,r}\|_{B_p} =  \biggl(\sum_{n=q}^r \frac{1}{n^p}\biggr)^{\frac1p}. \end{equation}
The right-hand side of~\eqref{normypq}  is certainly a lower bound on the norm of~$y_{q,r}$ because the sets  \mbox{$M_q<M_{q+1}<\cdots<M_r$} are admissible, so that 
\begin{equation}\label{normypqlower} \| y_{q,r}\|_{B_p}\geqslant \nu_p(y_{q,r}; M_q,\ldots, M_r) = \biggl(\sum_{n=q}^r \frac{1}{n^p}\biggr)^{\frac1p}. \end{equation}

The reverse inequality requires more work. We begin by observing that  since $y_{q,r}$ is finitely supported, the supremum in the definition~\eqref{defnBaernsteinnorm} of the norm of~$y_{q,r}$ is attained, say 
\begin{equation}\label{normofyqrattained} \| y_{q,r}\|_{B_p} = \nu_p(y_{q,r}; N_1,\ldots,N_k), \end{equation}
where $k\in\N$ and $N_1<\cdots<N_k$ are admissible. We may suppose that $\min N_1\geqslant m_q$ and $\max N_k< m_{r+1}$ because the support of~$y_{q,r}$ is contained in $[m_q,m_{r+1})$. We shall show that this additional assumption forces $(N_1,\ldots, N_k) = (M_q,\ldots,  M_r)$, which in the light of~\eqref{normypqlower} will ensure that~\eqref{normypq} holds.\smallskip 

First, we note that $\min N_1 = m_q$ because otherwise we would have $\{m_q\}<N_1$ and 
\[ \nu_p(y_{q,r};N_1,\ldots,N_k)< \nu_p(y_{q,r};\{ m_q\}, N_1,\ldots,N_k)\leqslant \| y_{q,r}\|_{B_p}, \] 
contrary to~\eqref{normofyqrattained}. Similar reasoning shows that 
\[ \max N_k = m_{r+1}-1\qquad\text{and}\qquad \max N_{j} +1 = \min N_{j+1}\qquad (j\in\{1,\ldots, k-1\}). \] 
Moreover,  each of the sets $N_1,\ldots,N_k$ must be an inter\-val:  indeed,
if~$N_j$ were not an~inter\-val for some~$j\in\{1,\ldots,k\}$, then the unique interval~$N_j'$ with $\min N_j' = \min N_j$ and \mbox{$|N_j'| = |N_j|$} is admissible and satisfies  $\mu(y_{q,r},N_j')\geqslant \mu(y_{q,r},N_j)$ because the coordinates of~$y_{q,r}$ in~$[m_q,m_{r+1})$ are decreasing and positive, and hence
\[ \nu_p(y_{q,r};N_1,\ldots,N_k)< \nu_p(y_{q,r}; N_1,\ldots,N_{j-1},N_j',\{\max N_j'+1\},N_{j+1},\ldots, N_k)\leqslant \| y_{q,r}\|_{B_p}, \]
again contrary to~\eqref{normofyqrattained}. Thus we conclude that \[ \bigcup_{j=1}^k N_j = [m_q,m_{r+1})\cap\N = \bigcup_{s=q}^r M_s. \]

Assume towards a contradiction that $(N_1,\ldots,N_k)\ne (M_q,\ldots,M_r)$. For each $s\in\N$, $M_s$~is maximal among all admissible intervals with minimum~$m_s$, and so there must be some $j\in\{2,\ldots,k\}$ such that $\min N_j\notin\{ m_s : q<s\leqslant r\}$. Let~$j_0$ be the smallest such~$j$, so that $\min 
N_{j_0 - 1} = m_s$ for some $s\in\{q,\ldots,r-1\}$, but $t = \min N_{j_0}$ is not of the form~$m_u$ for any~$u\in\N$, which implies that $t<m_{s+1}$ because $N_{j_0-1}$ is admissible. 
Since $N_{j_0}$ is also admissible, we have $|N_{j_0}|\leqslant t$. Hence the interval
$N_{j_0}'' = N_{j_0}\cap [m_{s+1},\infty)\subseteq M_{s+1}$ satisfies
\[ |N_{j_0}''| = |N_{j_0}| - (m_{s+1} - t)\leqslant 2t - m_{s+1} = 2(t-m_s), \]
from which we deduce that
\begin{equation}\label{mueq} \mu(y_{q,r},N_{j_0}'') = \frac{|N_{j_0}''|}{(s+1)2^s}\leqslant \frac{2(t - m_s)}{(s+1)2^s}< 
\frac{t - m_s}{s2^{s-1}} = \frac{|N_{j_0-1}|}{s2^{s-1}} = \mu(y_{q,r},N_{j_0-1}). 
\end{equation}
Set $L = [t, m_{s+1})\cap\N$. Then we have $N_{j_0} = L\cup N_{j_0}''$ and $M_s = N_{j_0-1}\cup L$, with both unions being disjoint, so in the light of~\eqref{mueq} we can apply Lemma~\ref{calclemma} to obtain the following in\-equality
\begin{multline*} \mu(y_{q,r},N_{j_0-1})^p+ \mu(y_{q,r},N_{j_0})^p 
= \mu(y_{q,r},N_{j_0-1})^p + \bigl(\mu(y_{q,r},L) + \mu(y_{q,r}, N_{j_0}'')\bigr)^p\\
< \bigl(\mu(y_{q,r},N_{j_0-1}) + \mu(y_{q,r}, L)\bigr)^p + \mu(y_{q,r}, N_{j_0}'')^p = 
\mu(y_{q,r},M_s)^p + \mu(y_{q,r}, N_{j_0}'')^p. \end{multline*}
Thus the admissible sets $N_1<\cdots< N_{j_0-2} < M_s < N_{j_0}'' < N_{j_0+1}<\cdots <N_k$ satisfy
\[ \nu_p(y_{q,r};N_1,\ldots,N_k)< \nu(y_{q,r};N_1,\ldots, N_{j_0-2}, M_s, N_{j_0}'', N_{j_0+1},\ldots, N_k)\leqslant \|y_{q,r}\|_{B_p}, \]
once again contradicting~\eqref{normofyqrattained}. Therefore we must have $(N_1,\ldots,N_k) = (M_q,\ldots,M_r)$, and as already explained,~\eqref{normypq} follows.\smallskip  

Direct application of the definitions shows that
\begin{equation}\label{Akypq} \operatorname{diag}(U_n) y_{q,r} = \sum_{n=q}^r \frac{1}{2^{n-1}n} U_n\biggl(\sum_{j\in M_n} b_j\biggr) = \sum_{n=q}^r\frac{1}{n}b_{m_n}. \end{equation}
In particular, choosing $r = 2q$ and $q\geqslant 3$, we see that 
the support $N = \{ m_n\colon q\leqslant n\leqslant 2q\}$ of the vector on the right-hand side of~\eqref{Akypq} is admissible because \[ |N| = q+1\leqslant 2^{q-1} = \min N. \] Hence on the one hand we have
\[ \|\operatorname{diag}(U_n)y_{q,2q}\|_{B_p} = \mu(\operatorname{diag}(U_n)y_{q,2q},N)=\sum_{n=q}^{2q}\frac{1}{n}\geqslant 
\frac{q+1}{2q}> \frac12\qquad (q\geqslant 3). \]
On the other hand, \eqref{normypq} shows that $\| y_{q,2q}\|_{B_p}\to 0$ as $q\to\infty$ because the series $\sum_{n=1}^\infty 1/n^p$ converges. Therefore $\operatorname{diag}(U_n)$ cannot be bounded. 
\end{proof}

\section{Open problems}
\noindent
Pfitzner~\cite{pfitzner} showed that $C^*$-algebras have Pe\l{}czy\'{n}ski's property ($V$). This implies that von Neu\-mann alge\-bras are Grothen\-dieck spaces, so in particular~$\mathscr{B}(H)$ is a Grothen\-dieck space for each Hilbert space~$H$. No other examples of infinite-dimensional Banach spaces~$X$ for which~$\mathscr{B}(X)$ is a Grothen\-dieck space are known. In particular, the following questions re\-main open:  
\begin{itemize}
\item Is $\mathscr{B}(\ell_p)$ a Grothen\-dieck space for $p\in (1,\infty)\setminus\{2\}$? In this case, the problem is equivalent to deciding whether  $\bigl(\bigoplus_{n\in\N}\mathscr{B}(\ell_p^n)\bigr)_{\ell_\infty}$ is a Grothen\-dieck space because~$\mathscr{B}(\ell_p)$ and $\bigl(\bigoplus_{n\in\N}\mathscr{B}(\ell_p^n)\bigr)_{\ell_\infty}$ are isomorphic as Banach spaces (see \cite[Theorem~2.1]{AF}; for $p=2$, see also \cite[p.~317]{ChrSinc}, where the result is credited to Haage\-rup and Linden\-strauss).
\item Is $\mathscr{B}(X)$ a Grothen\-dieck space for every weak Hilbert space~$X$?
\item More generally, is $\mathscr{B}(X)$ a Grothen\-dieck space for every super-reflexive Banach space~$X$? 
\item Let $X$ be a reflexive Banach space with an unconditional finite-di\-men\-sional Schau\-der de\-compo\-si\-tion $(F_n)_{n=1}^\infty$ whose blocks $F_n$ are uniformly isomorphic to~$\ell_1^{m_n}$, where $m_n = \dim F_n\to\infty$ as $n\to\infty$. Is it true that~$\mathscr{B}(X)$ is not a Grothen\-dieck space? If it is true, it would provide a unified proof of the two cases considered in Theorem~\ref{mainthm}. 
\item Is there a non-reflexive Banach space~$X$ for which~$\mathscr{B}(X)$ is a Grothen\-dieck space? If such a space~$X$ exists, both~$X$ and~$X^*$ would necessarily be Grothendieck spaces (because~$\mathscr{B}(X)$ contains complemented copies of them), and therefore~$X$ would be non-separable. No non-reflexive Grothendieck space~$X$ for which~$X^*$ is also a Grothendieck space is known.
\end{itemize}

\subsection*{Acknowledgements} The work that led to the present paper was initiated when the first-named author visited the UK in February/March 2017. We are pleased to acknowledge the award of a Scheme~2 grant from the London Mathematical Society that made this visit possible. The second-named author acknowledges with thanks funding received from the European Research Council (ERC Grant Agreement No.~291497). Finally, we are grateful to Professor Hermann Pfitzner for having pointed out a critical error in an earlier version of this paper.

\end{document}